\documentclass{article}
\usepackage{amsfonts}

\newtheorem{theorem}{Theorem}
\newtheorem{cor}[theorem]{Corollary}

\newtheorem{lem}[theorem]{Lemma}
\newtheorem{Exa}[theorem]{Example}

\newenvironment{proof}[1][Proof]{\textbf{#1.} }{\ \rule{0.5em}{0.5em}}
\title{Singularities in Space-time Foam Algebras}
\author{Jan Harm van der Walt\\
Department of Mathematics and Applied Mathematics\\
University of Pretoria\\
Pretoria 0002}

\date{}

\begin{document}

\maketitle

\begin{abstract}
\noindent In this paper we consider the structure of the
singularity sets associated with generalized functions in certain
space-time foam algebras of generalized functions.  In particular,
we consider the algebra that is defined in terms of an asymptotic
vanishing condition outside sets of first Baire category.  It is
shown that this algebra is in fact isomorphic to the earlier
closed nowhere dense algebras.  We also discuss general questions
regarding singularities that can be handled in space-time foam
algebras.
\end{abstract}
\vspace{0.2cm}

\vspace{0.2cm}

\vspace{0.2cm}


\section{Introduction}

It is well known that the linear theory of distributions is not
suited to a comprehensive and systematic treatment of nonlinear
problems. This inability of $\mathcal{D}'$ distributions to
accommodate nonlinear phenomena is exemplified by the so called
Schwarz Impossibility \cite{Schwarz 1}.  One formulation of
Schwarz's result is as follows:  There is no symmetric bilinear
mapping
\begin{eqnarray}
\star:\mathcal{D}'\left(\Omega\right) \times \mathcal{D}'
\left(\Omega\right) \ni \left(S,T\right) \rightarrow S\star T\in
\mathcal{D}'\left(\Omega\right)\nonumber
\end{eqnarray}
so that, $S\star T$ is the usual pointwise product of continuous
functions, when $S,T\in\mathcal{C}^{0}\left(\Omega\right)$. That
is, $\mathcal{D}'\left(\Omega\right)$ is not closed under any
multiplication that extends the multiplication of continuous
functions.

One way in which one may overcome the mentioned impossibility is
to embed $\mathcal{D}'\left(\Omega\right)$ as a vector subspace in
$\mathcal{A}\left(\Omega\right)$, where
$\mathcal{A}\left(\Omega\right)$ is a quotient algebra
\begin{eqnarray}
\mathcal{A}\left(\Omega\right) =
\mathcal{S}/\mathcal{I},\label{QAlgebra}
\end{eqnarray}
with $\mathcal{S}$ a subalgebra in $\mathcal{C}^{\infty}\left(
\Omega\right)^{\Lambda}$, for some index set $\Lambda$, and
$\mathcal{I}$ an ideal in $\mathcal{S}$.  This approach was
initiated by Rosinger \cite{Rosinger 4, Rosinger 5}, and developed
further in \cite{Rosinger 6,Rosinger 7,Rosinger 8,Rosinger 9}.  A
important particular case of the theory was introduced
independently by Colombeau \cite{ColombeauI, ColombeauII}.  This
version of the theory has seen rapid development and a variety of
applications over the past three decades, see for instance
\cite{Adamczewski, Bernard et al,Colombeau and Heibig,Colombeau et
al,Grosser et al I,Obergugenberger}.  One of the major advantages
of the (full) Colombeau algebra $\mathcal{G}\left(\Omega\right)$
is the rather natural close connection with distributions.  In
particular, $\mathcal{G}\left(\Omega\right)$ allows a {\it
canonical} linear embedding of the space of distributions.
Furthermore, the partial derivatives in
$\mathcal{G}\left(\Omega\right)$ coincide with the distributional
derivatives, when restricted to $\mathcal{D}'\left(\Omega\right)$.

Despite the mentioned utility of $\mathcal{G}\left(\Omega\right)$,
several deep results obtained within the more general version of
the theory appear to have no counterpart in the Colombeau setting.
In this regard, we may recall the global version of the
Cauchy-Kovalevskaia Theorem \cite{Rosinger 8,Rosinger 9}, obtained
in the so called closed nowhere dense algebras.  In
contradistinction with the closed nowhere dense algebra, in
$\mathcal{G}\left(\Omega\right)$ one cannot formulate, let alone
prove, such a global existence result \cite[Section 3.1]{Rosinger
14}. Indeed, due to the polynomial type growth conditions imposed
on generalized functions in $\mathcal{G}\left( \Omega\right)$ near
singularities, the class of nonlinear operations that can be
defined on $\mathcal{G}\left(\Omega\right)$, and the types of
singularities that generalized functions in
$\mathcal{G}\left(\Omega\right)$ can handle, are severely limited.
In particular, $\mathcal{G}\left(\Omega\right)$ fails to be a {\it
flabby sheaf}.

As mentioned, due to the failure of
$\mathcal{G}\left(\Omega\right)$ to be a flabby sheaf, the class
of singularities that this algebra can deal with is rather
limited.  Furthermore, the definition of Colombeau algebras on
manifolds are severely complicated by this failure, see for
instance \cite{Grosser et al I}.

Recently, Rosinger \cite{Rosinger 12,Rosinger 13,Rosinger 14}
introduced a general class of differential algebras, namely, the
space-time foam (STF) algebras, which include the earlier nowhere
dense algebras as a particular case.  Each such algebra admits a
linear embedding of $\mathcal{D}'\left( \Omega\right)$, and
contains $\mathcal{C}^{\infty}\left(\Omega \right)$ as a
subalgebra.  However, the embedding of distributions into STF
algebras is in general not canonical.  In fact, for any such
algebra, there may be infinitely many linear embeddings of
$\mathcal{D}'\left(\Omega\right)$, none of which are to be
preferred above the others.  Thus, as far as the embedding of
distributions in to differential algebras are concerned, the
Colombeau algebra appears a more natural setting than the STF
algebras.

Here we may note that a large class of algebras admitting a
canonical linear embedding of distributions, which commutes with
distributional derivatives, were recently introduced by Vernaeve
\cite{Vernaeve}.  These algebras also admits a global version of
the Cauchy-Kovalevskaia Theorem.

On the other hand, the STF algebras are able to handle a far
larger class of singularities than the Colombeau algebras, and for
that matter, any other differential algebra introduced so far.
Indeed, STF algebras can deal with singularities that occur on any
set $\Sigma$, provided only that $\Omega\setminus\Sigma$ is dense.
Also, STF algebras form a fine and flabby sheaf.  The mentioned
flabby sheaf structure of STF algebras results in a simple
construction of generalized functions on manifolds \cite{Rosinger
13}, although this fact is often not fully appreciated.

The utility of the STF algebras is further demonstrated by the
Global version of the Cauchy-Kovalevskaia Theorem, which holds in
suitable STF algebras \cite{Rosinger 14}, which is a strengthening
of the previously known global existence result for analytic PDEs
\cite{Rosinger 9}.  Besides this global existence result for large
classes of nonlinear PDEs, STF algebras have seen applications in
a variety of fields, including abstract differential geometry, de
Rham cohomology and quantum gravity, see for instance
\cite{Mallios and Rosinger}.

In this paper we investigate the structure of singularities in
certain STF algebras.  In particular, the Baire I algebras
$\mathcal{B}_{L,B-I}\left(\Omega\right)$, where $L$ is an
appropriate index set, are investigated, and compared with the
nowhere dense algebra.  In this regard, we show that these two
algebras are, in many cases, identical.  We further investigate
the extent to which large classes of singularities may be
accommodated in STF algebras.

The paper is organized as follows.  In Section 2 we consider the
relation between the Baire I algebras and chains of closed nowhere
dense algebras.  In particular, Section 2.1 recalls the basic
construction of nowhere dense algebras, while Section 2.2 is
concerned with the more general STF algebras.  Finally, it is
shown that the Baire I algebra is in fact, in many cases,
identical with the nowhere dense algebra.  A more general
discussion of singularities in STF algebras is presented in
Section 2.5 and 2.6.

\section{Baire I Algebras vs Closed Nowhere Dense Algebras}

In this section we discuss certain STF algebras, recently
introduced by Rosinger \cite{Rosinger 12,Rosinger 13,Rosinger 14}.
In particular, we consider the algebra $\mathcal{B}_{L,B-I}
\left(\Omega\right)$, which is defined in terms of a dense
vanishing condition off sets of first Baire category.  These
algebras, as part of the larger family of STF algebras, appear to
have a rather clear and natural structure. Indeed, as mentioned,
STF algebras form flabby sheaves, hence the following two usefula
nd remarkable properties of such algebras:
\begin{description}
    \item[(a)] STF algebras can handle large classes of
    singularities
    \item[(b)] STF algebras are defined on general smooth
    manifolds in a straight forward and simple way.
\end{description}
Furthermore, certain STF algebras, including as a particular case
$\mathcal{B}_{L,B-I}\left(\Omega\right)$, admit a global version
of the Cauchy-Kovalevskaia Theorem.

We show that, for a large class of index sets $L$, the algebra
$\mathcal{B}_{L,B-I}\left( \Omega \right)$ is in fact isomorphic
to the nowhere dense algebra
$\mathcal{A}^{\infty}_{L,nd}\left(\Omega\right)$, see
\cite{Rosinger 7,Rosinger 8,Rosinger 9}.  This follows from a
generalization of a suitable Baire category argument \cite[Chapter
3, Appendix 1]{Rosinger 9}

\subsection{Nowhere Dense Algebras}

The closed nowhere dense algebras were first introduced by
Rosinger \cite{Rosinger 7}, and are designed in such a way as to
accommodate certain singular generalized solutions of PDEs, such
as shock waves, as actual solutions of the respective equations in
a differential-algebraic sense.  Here we briefly recall the
construction of these algebras, and some of their basic features.

In this regard, let $\Omega\subseteq \mathbb{R}^{n}$ be a non
void, open set.  Let $L=\left(\Lambda,\leq\right)$ be an infinite
right directed index set.  That is,
\begin{eqnarray}
\begin{array}{ll}
\forall & \lambda,\mu\in\Lambda~: \\
\exists & \nu\in\Lambda~: \\
& \lambda,\mu\leq \nu \\
\end{array}.\nonumber
\end{eqnarray}
With respect to the usual componentwise operations,
$\mathcal{C}^{\infty}\left(\Omega\right)^{\Lambda}$ is a unital
and commutative algebra, and the set
\begin{eqnarray}
\mathcal{I}^{\infty}_{L,nd}\left(\Omega\right) = \left\{
w=\left(w_{\lambda}\right)_{\lambda\in\Lambda}~\begin{array}{|ll}
\exists & \Gamma\subset\Omega\mbox{ closed nowhere dense}~: \\
\forall & x\in\Omega\setminus\Gamma~: \\
\exists & \lambda\in\Lambda~: \\
\forall & \mu\in\Lambda,~\mu\geq \lambda~: \\
& w_{\mu}\left(x\right)=0 \\
\end{array} \right\}\label{IndkDef}
\end{eqnarray}
is an ideal in
$\mathcal{C}^{\infty}\left(\Omega\right)^{\Lambda}$. Based on a
Baire category argument in $\mathbb{R}^{n}$, see \cite[Chapter 2,
Appendix 1]{Rosinger 9} the condition in the definition of
$\mathcal{I}^{\infty}_{L,nd}\left(\Omega\right)$ is equivalent to
\begin{eqnarray}
\begin{array}{ll}
\exists & \Gamma\subset\Omega\mbox{ closed nowhere dense}~: \\
\forall & x\in\Omega\setminus\Gamma~: \\
\exists & \lambda\in\mathbb{N},~ V\subseteq \Omega\setminus\Gamma~\mbox{a neighborhood of }x~: \\
\forall & y\in V,~\mu\in\Lambda,~\mu\geq \lambda~: \\
& w_{\mu}\left(y\right)=0 \\
\end{array},\label{IndkEqDefI}
\end{eqnarray}
which is easily seen to be equivalent to
\begin{eqnarray}
\begin{array}{ll}
\exists & \Gamma\subset\Omega\mbox{ closed nowhere dense}~: \\
\forall & x\in\Omega\setminus\Gamma~: \\
\exists & \lambda\in\mathbb{N},~ V\subseteq \Omega\setminus\Gamma~\mbox{a neighborhood of }x~: \\
\forall & y\in V,~p\in\mathbb{N}^{n},~\mu\in\Lambda,~\mu\geq \lambda~: \\
& D^{p}w_{\mu}\left(y\right)=0 \\
\end{array}.\label{IndkEqDefII}
\end{eqnarray}
In view of (\ref{IndkEqDefII}), it follows immediately that the
differential operators
\begin{eqnarray}
D^{p}:\mathcal{C}^{\infty}\left(\Omega\right)^{\Lambda}\ni
w=\left(w_{n}\right)\mapsto D^{p}w=\left(D^{p}w_{n}\right)\in
\mathcal{C}^{\infty}\left(\Omega\right)^{\Lambda}\nonumber
\end{eqnarray}
satisfy the inclusion
\begin{eqnarray}
D^{p}\left(\mathcal{I}^{\infty}_{L,nd}\left(\Omega\right)\right)\subseteq
\mathcal{I}^{\infty}_{L,nd}\left(\Omega\right),~p\in\mathbb{N}^{n}.
\end{eqnarray}
Thus, the usual partial derivative operators on
$\mathcal{C}^{\infty}\left(\Omega\right)$ extend to mappings
\begin{eqnarray}
D^{p}:\mathcal{A}^{\infty}_{L,nd}\left(\Omega\right)\ni w +
\mathcal{I}^{\infty}_{L,nd}\left(\Omega\right) \mapsto D^{p}w+
\mathcal{I}^{\infty}_{L,nd}\left(\Omega\right) \in
\mathcal{A}^{\infty}_{L,nd}\left(\Omega\right),~p\in\mathbb{N}^{n},\label{DpAndkDef}
\end{eqnarray}
where $\mathcal{A}^{\infty}_{L,nd}\left(\Omega\right)$ is the
quotient algebra
\begin{eqnarray}
\mathcal{A}^{\infty}_{L,nd}\left(\Omega\right) =
\mathcal{C}^{\infty} \left(
\Omega\right)^{\mathbb{N}}/\mathcal{I}_{L,nd}^{\infty}\left(\Omega
\right).\nonumber
\end{eqnarray}

In view of (\ref{IndkEqDefI}) it is clear that the ideal
$\mathcal{I}^{\infty}_{L,nd}\left(\Omega\right)$ satisfies the
{\it off diagonal} or {\it neutrix condition}
\begin{eqnarray}
\mathcal{I}^{\infty}_{L,nd}\left(\Omega\right) \cap
\mathcal{U}^{\infty}_{\Lambda}\left(\Omega\right)=\{0\},\label{Neutrix}
\end{eqnarray}
where $\mathcal{U}^{\infty}_{\Lambda}\left(\Omega\right) = \left\{
u\left(\psi\right)
=\left(\psi_{\lambda}\right)_{\lambda\in\Lambda}~|~\psi_{n}=\psi,~\lambda\in\Lambda
\right\}$ is the diagonal in
$\mathcal{C}^{\infty}\left(\Omega\right)^{\Lambda}$.  As such, see
\cite[Chapter 6]{Rosinger 9}, the algebras
$\mathcal{A}^{\infty}_{nd} \left(\Omega\right)$ contains
$\mathcal{C}^{\infty}\left(\Omega\right)$ as a subalgebra.  That
is, there exists a canonical, injective algebra homomorphism
\begin{eqnarray}
\mathcal{C}^{\infty}\left(\Omega\right)\hookrightarrow
\mathcal{A}^{\infty}_{L,nd}\left(\Omega\right)\label{CInftyEmbedding}
\end{eqnarray}
In particular, the embedding (\ref{CInftyEmbedding}) is an {\it
embedding of differential algebras}.  That is, for all
$p\in\mathbb{N}^{n}$, the diagram
\begin{eqnarray}
\setlength{\unitlength}{1cm} \thicklines
\begin{array}{l}
\begin{picture}(14,3.5)

\put(3.5,2.5){$\mathcal{C}^{\infty}\left(\Omega\right)$}
\put(4.7,2.6){\vector(1,0){5}}
\put(9.9,2.5){$\mathcal{C}^{\infty}\left(\Omega\right)$}
\put(6.9,2.8){$D^{p}$}

\put(3.9,2.35){\vector(0,-1){1.45}}
\put(10.3,2.35){\vector(0,-1){1.45}}
\put(3.5,1.55){$\hookrightarrow$}
\put(10.4,1.55){$\hookrightarrow$}

\put(3.2,0.5){$\mathcal{A}^{\infty}_{L,nd}\left(\Omega\right)$}
\put(4.7,0.6){\vector(1,0){4.9}}
\put(9.7,0.5){$\mathcal{A}^{\infty}_{L,nd}\left(\Omega\right)$}
\put(6.9,0.8){$D^{p}$}

\end{picture}
\end{array}\label{CInftyDAEmbedding}
\end{eqnarray}
commutes.  The embedding (\ref{CInftyEmbedding}) preserves not
only the algebraic structure of
$\mathcal{C}^{\infty}\left(\Omega\right)$, but also its
differential structure.

The neutrix condition (\ref{Neutrix}) implies also the existence
of an injective, linear mapping
\begin{eqnarray}
\mathcal{D}'\left(\Omega\right)\hookrightarrow
\mathcal{A}^{\infty}
_{L,nd}\left(\Omega\right).\label{DistEmbedding}
\end{eqnarray}
That is, the differential algebra $\mathcal{A}^{\infty}
_{L,nd}\left(\Omega\right)$ contains the distributions as a linear
subspace, see \cite[pp. 234--244]{Rosinger 9} where those algebras
that admit linear embeddings of distributions are characterized in
terms such off diagonal conditions. However, in contradistinction
with (\ref{CInftyEmbedding}), the embedding (\ref{DistEmbedding})
does not commute with partial derivatives.  In other words, the
differential operators on
$\mathcal{A}^{\infty}_{L,nd}\left(\Omega\right)$ do not, in
general, coincide with distributional derivatives, when restricted
to $\mathcal{D}'\left(\Omega\right)$.

\subsection{STF Algebras}

Proceeding from the particular to the general, let us now recall
the construction of the large class of differential algebras,
namely, the STF algebras, first introduced in \cite{Rosinger 12,
Rosinger 13, Rosinger 14}.  In this regard, let $\mathcal{S}$ be a
collection of subsets of $\Omega$ that satisfies the conditions
\begin{eqnarray}
\begin{array}{ll}
\forall & \Sigma\in \mathcal{S}~: \\
& \Omega\setminus\Sigma~{\rm is~dense~in}~\Omega \\
\end{array}\label{DenseComplement}
\end{eqnarray}
and
\begin{eqnarray}
\begin{array}{ll}
\forall & \Sigma,~\Sigma'\in\mathcal{S}~: \\
\exists & \Sigma''\in\mathcal{S}~: \\
& \Sigma\cup\Sigma'\subseteq \Sigma'' \\
\end{array}\label{UnionProp}
\end{eqnarray}
For each singularity set $\Sigma\in\mathcal{S}$ we denote by
$\mathcal{J}_{L,\Sigma}\left(\Omega\right)$ the ideal in
$\mathcal{C}^{\infty}\left(\Omega\right)^{\Lambda}$ of all
$\Lambda$-sequences $w=\left(w_{\lambda}\right)_{\lambda\in
\Lambda}$ that, outside the singularity set $\Sigma$ will satisfy
the {\it asymptotic vanishing} condition
\begin{eqnarray}
\begin{array}{ll}
\forall & x\in\Omega\setminus\Sigma~: \\
\exists & \lambda\in\Lambda~: \\
\forall & p\in\mathbb{N}^{n},~\mu\in\Lambda,~\mu\geq\lambda~: \\
& D^{p}w_{\mu}\left(x\right)=0 \\
\end{array}.\label{AssVanish}
\end{eqnarray}
It is easily seen that
\begin{eqnarray}
\Sigma\subseteq\Sigma'\Rightarrow
\mathcal{J}_{L,\Sigma}\left(\Omega\right) \subseteq
\mathcal{J}_{L,\Sigma'}\left(\Omega\right)\nonumber
\end{eqnarray}
so that it follows by (\ref{UnionProp}) that
\begin{eqnarray}
\mathcal{J}_{L,\mathcal{S}}\left(\Omega\right) =
\bigcup_{\Sigma\in\mathcal{S}}\mathcal{J}_{L,\Sigma}\left(\Omega\right)\label{IdealDef}
\end{eqnarray}
is also an ideal in $\mathcal{C}^{\infty}\left(\Omega\right)^{
\Lambda}$.  The {\it space-time foam algebra} associated with
$\mathcal{S}$ and $L$ is now defined as
\begin{eqnarray}
\mathcal{B}_{L,\mathcal{S}}\left(\Omega\right) =
\mathcal{C}^{\infty}\left(\Omega\right)^{\Lambda}/
\mathcal{J}_{L,\mathcal{S}}\left(\Omega\right)\label{STFAlgDef}
\end{eqnarray}
It is clear that the ideal
$\mathcal{J}_{L,\mathcal{S}}\left(\Omega\right)$ satisfies the off
diagonal condition
\begin{eqnarray}
\mathcal{J}_{L,\mathcal{S}}\left(\Omega\right)\cap
\mathcal{U}^{\infty}_{\Lambda}\left(\Omega\right)=\{0\},\nonumber
\end{eqnarray}
as well as the inclusion
\begin{eqnarray}
D^{p}\left(\mathcal{J}_{L,\mathcal{S}}\left(\Omega\right)\right)
\subseteq \mathcal{J}_{L,\mathcal{S}}\left(\Omega\right),~
p\in\mathbb{N}.\nonumber
\end{eqnarray}
Consequently, $\mathcal{B}_{L,\mathcal{S}}\left(\Omega\right)$
contains $\mathcal{C}^{\infty}\left(\Omega\right)$ as a
subalgebra, and the differential operators on
$\mathcal{C}^{\infty}\left(\Omega\right)$ extend to
$\mathcal{B}_{L,\mathcal{S}}\left(\Omega\right)$.  Furthermore, as
in the case of the nowhere dense algebras,
$\mathcal{B}_{L,\mathcal{S}}\left(\Omega\right)$ contains
$\mathcal{D}'\left(\Omega\right)$ as a linear subspace.

The nowhere dense algebras $\mathcal{A}^{\infty}_{L,nd}\left(
\Omega\right)$ is clearly but a particular case of the STF
algebras. Indeed, if the family of singularity sets $\mathcal{S}$
is chosen to be the collection of all closed nowhere dense subsets
of $\Omega$, the resulting STF algebra is precisely the algebra
$\mathcal{A}^{\infty}_{L,nd}\left( \Omega\right)$.

The nowhere dense algebras $\mathcal{A}^{\infty}_{L,nd}
\left(\Omega\right)$ allow singularities which may occur on
arbitrary closed nowhere dense sets.  As such, these singularity
sets are far lager than those that can be dealt with, say, though
distributions.  In particular, such a singularity set my have
arbitrary large positive Lebesgue measure \cite{Oxtoby}.  However,
general STF algebras are able to handle singularities on even
bigger sets.  Indeed, due to the requirement
(\ref{DenseComplement}), any set $\Sigma\subset \Omega$ may act as
the singularity set of a generalized function, provided only that
its complement $\Omega\setminus\Sigma$ is dense in $\Omega$. In
this way, the cardinality of the singularity set of a generalized
function may turn out to be greater than that of the set of its
regular points.

\subsection{Baire I Algebras}

Another important particular case of the SFT construction, given
in Section 2.2, is the so called Baire I algebras
$\mathcal{B}_{L,B-I}\left( \Omega \right)$.  In this case, the
family $\mathcal{S}_{B-I}$ of singularity sets is chosen as
\begin{eqnarray}
\mathcal{S}_{B-I} = \left\{ \Sigma\subset \Omega ~:~
\Sigma~\mbox{is first Baire category}\right\}.\nonumber
\end{eqnarray}
Then the ideal $\mathcal{J}_{L,\mathcal{S}_{B-I}}
\left(\Omega\right) = \mathcal{J}_{L,B-I}\left( \Omega\right)$
consists of those $\Lambda$-sequences of smooth functions
$w=(w_{\lambda})_{\lambda\in\Lambda}$ that satisfy
\begin{eqnarray}
\begin{array}{ll}
\exists & \Sigma\subset\Omega~\mbox{of first Baire category}~: \\
\forall & x\in\Omega\setminus\Sigma~: \\
\exists & \lambda\in\Lambda~: \\
\forall & p\in\mathbb{N}^{n},~\mu\in\Lambda,~\mu\geq \lambda~: \\
& w_{\mu}(x)=0 \\
\end{array}\label{BIIdeal}
\end{eqnarray}
Since each closed nowhere dense set $\Gamma\subseteq \Omega$ is of
first Baire category, it is clear that the inclusion
\begin{eqnarray}
\mathcal{I}^{\infty}_{L,nd}\left(\Omega\right) \subseteq
\mathcal{J}_{L,B-I}\left( \Omega\right)\label{InftyIBIInclusion}
\end{eqnarray}
holds.  Consequently, the differential algebras
$\mathcal{B}_{L,B-I}\left( \Omega \right)$ and
$\mathcal{A}^{\infty}_{L,nd}\left( \Omega \right)$ are related
though the existence of a canonical, surjective algebra
homomorphism
\begin{eqnarray}
\mathcal{A}^{\infty}_{L,nd}\left( \Omega \right)\rightarrow
\mathcal{B}_{L,B-I}\left( \Omega \right).\label{AInftyBIHom}
\end{eqnarray}
In view of the existence of such an algebra homomorphism, the
generalized functions in $\mathcal{B}_{L,B-I}\left( \Omega
\right)$ may be considered to be more regular than those in
$\mathcal{A}^{\infty}_{L,nd}\left( \Omega \right)$, see
\cite[Sect. 1.7]{Rosinger 14} and \cite{vdWalt}.

\subsection{Baire I Algebras are Nowhere Dense Algebras}

As mentioned, the Baire I algebras are related to the nowhere
dense algebras through the existence of a surjective algebra
homomorphism (\ref{AInftyBIHom}).  We now proceed to show that,
for a large class of index sets $L$, this algebra homomorphism is
in fact an isomorphism.
\begin{lem}\label{Lemma1}
Suppose that $X$ is a Baire space, and $Y$ any topological space.
Assume that the sequence $\left(w_{\lambda}\right)\subset
\mathcal{C}\left(X,Y\right)$ and the continuous function
$w:X\rightarrow Y$ satisfy
\begin{eqnarray}
\begin{array}{ll}
\forall & x\in X~: \\
\exists & \lambda\in\mathbb{N}~: \\
\forall & \mu\in\mathbb{N},~\mu\geq \lambda~: \\
& w_{\mu}(x)=w(x) \\
\end{array}\nonumber
\end{eqnarray}
Then for every non void, open subset $A$ of $X$, there exists a
non void, relatively open subset $U$ of the closure $\overline{A}$
of $A$, and $\nu\in\mathbb{N}$ such that
\begin{eqnarray}
\begin{array}{ll}
\forall & x\in U~: \\
\forall & \lambda\in\mathbb{N},~\lambda\geq\nu~: \\
& w_{\lambda}(x)=w(x) \\
\end{array}\nonumber
\end{eqnarray}
\end{lem}
\begin{proof}
We claim that $\overline{A}$ is a Baire space.  In this regard,
for each $\lambda\in\mathbb{N}$, let $D_{\lambda}\subseteq
\overline{A}$ be relatively open and dense.  We may assume,
without loss of generality, that $D_{\lambda}\subseteq A$. Then,
for each $\lambda\in\mathbb{N}$, the set
$D_{\lambda}'=\left(X\setminus \overline{A}\right)\cup
D_{\lambda}$ is open and dense in $X$. Since $X$ is a Baire space,
it follows that the set
\begin{eqnarray}
\bigcap_{\lambda\in\mathbb{N}}D_{\lambda}'\nonumber
\end{eqnarray}
is dense in $X$.  Since $A$ is a non void and open, it follows
that
\begin{eqnarray}
\bigcap_{\lambda\in\mathbb{N}}D_{\lambda}'=\overline{A}\bigcap
\left(\bigcap_{\lambda\in\mathbb{N}}D_{\lambda}'\right)\neq
\emptyset. \nonumber
\end{eqnarray}
Thus $\overline{A}$ is a Baire space.\\
For every $\mu\in\mathbb{N}$, let us denote by $I_{\lambda}$ the
set
\begin{eqnarray}
A_{\mu} = \left\{x\in \overline{A}~\begin{array}{|ll}
\forall & \lambda\in\mathbb{N},~\lambda\geq \mu \\
& w_{\lambda}(x)=w(x) \\
\end{array}\right\}\nonumber
\end{eqnarray}
We clearly have $\overline{A} =
\bigcup_{\mu\in\mathbb{N}}A_{\mu}$.  Since $w_{\lambda}$, with
$\lambda\in\mathbb{N}$, and $w$ are all continuous functions, it
follows that each of the sets $A_{\mu}$ is a closed set in
$\overline{A}$.  Since $\overline{A}$ is a Baire space, it follows
that at least one of the sets $A_{\mu}$ must have nonempty
interior, relative to $\overline{A}$. This completes the proof.
\end{proof}

We call the right directed set $L=\left(\Lambda,\leq\right)$ {\it
countably co-final} if there exists a countable set
$\{\lambda_{i}~:~i\in\mathbb{N}\}\subseteq \Lambda$ such that the
mapping
\begin{eqnarray}
\mathbb{N}\ni i\mapsto \lambda_{i}\in\Lambda\nonumber
\end{eqnarray}
is an order isomorphic embedding, that is,
\begin{eqnarray}
i\leq j \Leftrightarrow \lambda_{i}\leq\lambda_{j},\nonumber
\end{eqnarray}
and
\begin{eqnarray}
\begin{array}{ll}
\forall & \lambda\in\Lambda~: \\
\exists & i\in\mathbb{N}~: \\
& \lambda\leq \lambda_{i} \\
\end{array}\nonumber
\end{eqnarray}
\begin{theorem}\label{AndFBCEq}
Suppose that $L$ is countably co-final.  Then
$\mathcal{J}_{L,B-I}\left( \Omega\right) =
\mathcal{I}^{\infty}_{L,nd}\left( \Omega\right)$.
\end{theorem}
\begin{proof}
We only give a proof in case $L=\mathbb{N}$.  The extension to
general countably co-finite index sets is obvious.\\
Since we already have the inclusion (\ref{InftyIBIInclusion}), it
is sufficient to show that
\begin{eqnarray}
\mathcal{J}_{\mathbb{N},B-I}\left( \Omega\right)\subseteq
\mathcal{I}^{\infty}_{\mathbb{N},nd}\left(\Omega\right).\label{JBIINDInc}
\end{eqnarray}
In this regard, consider any $w=\left(w_{\lambda}\right) \in
\mathcal{J}_{\mathbb{N},B-I}\left(\Omega\right)$.  Let
$\Sigma\subseteq \Omega$ be the set of first Baire category
associated with $w$ through (\ref{BIIdeal}).  Since $\Omega$ is a
Baire space, and $\Sigma$ is of first Baire category, it follows
that $E=\Omega \setminus\Sigma$ is also a Baire space.
Consequently, we may apply Lemma \ref{Lemma1} to the sequence
$\left(w_{n}\right)$, with the limiting function $w$ identically
$0$, to obtain the following:
\begin{eqnarray}
\begin{array}{ll}
\forall & x\in\Omega\setminus\Sigma~: \\
\exists & V\ni x~\mbox{ an open set in }\Omega,~\lambda\in\mathbb{N}~: \\
\forall & y\in V\cap \left(\Omega\setminus\Sigma\right),~ \mu\in\mathbb{N},~\mu\geq\lambda~: \\
& w_{\mu}(y)=0 \\
\end{array}\label{T1E1}
\end{eqnarray}
Since $\Omega\setminus\Sigma$ is dense in $\Omega$, by virtue of
$\Sigma$ being of a set of first Baire category, it follows from
(\ref{T1E1}) and the continuity of each $w_{\lambda}$ that
\begin{eqnarray}
\begin{array}{ll}
\forall & y\in V~: \\
\forall & \mu\in\mathbb{N},~\mu\geq \lambda~: \\
& w_{\mu}(y)=0 \\
\end{array}\nonumber
\end{eqnarray}
It again follows from the denseness of $\Omega\setminus\Sigma$
that the sequence $w$ satisfies (\ref{IndkEqDefI}).  Therefore
$w\in \mathcal{I}^{\infty}_{\mathbb{N},nd}\left(\Omega\right)$,
which verifies (\ref{JBIINDInc}).  This completes the proof.
\end{proof}
\begin{cor}\label{COR}
If $L$ is countably co-final, then $\mathcal{B}_{L,B-I}\left(
\Omega\right) = \mathcal{A}^{\infty}_{L,nd}\left( \Omega\right)$.
\end{cor}

Theorem \ref{AndFBCEq} and Corollary \ref{COR} may be interpreted
as follows: Within the setting of STF algebras of generalized
functions, a singularity that occurs on a set of first Baire
category will inevitably collapse to a, possibly smaller, closed
nowhere dense set.  Thus, closed nowhere dense singularity sets
are maximal within all singularity sets of first Baire category.

However, one of the main reasons for the use of STF algebras is
their ability to handle singularities on large, in particular
dense sets. Therefore Corollary \ref{COR} raises two important
questions: Firstly, are there convenient STF algebras that can in
fact handle singularities, even on dense sets?  Secondly, what are
the possible limitations on the size of singularity sets that the
STF algebras can handle?  In the following two sections we address
these two issues.  The former of these two questions is answered
in full in Section 2.6.  The latter is more subtle, and we give
only a brief discussion of the issues involved.

\subsection{Limitations on the Size of Singularity Sets}

It is a rather straight forward matter to construct STF algebras
that can handle singularities on large singularity sets, provided
that the family of admissible singularity sets is small.  In
particular, one may chose the family $\mathcal{S}$ in the
construction of STF algebras presented in Section 2.2 to consist
of a single set $\Sigma\subset\Omega$ whose complement
$\Omega\setminus\Sigma$ is countable and dense.  In this case the
singularity set of a generalized functions in
$\mathcal{B}_{L,\mathcal{S}}\left(\Omega \right)$ has larger
cardinality than the set of singular points.  In particular, the
singularity set $\Sigma$ is uncountable, while the set of singular
points $\Omega\setminus\Sigma$ is countable.  However, the algebra
$\mathcal{B}_{L,\mathcal{S}}\left(\Omega\right)$ can only handle
singularities that occur on one single set, namely, the set
$\Sigma$.  Hence this algebra is rather limited.

On the other hand, the use of large, and in particular infinite,
families of singularity sets appears to place some limitations on
the possible size of the singularity sets involved. In this
regard, consider any family of singularity sets $\mathcal{S}$ that
contains $\mathcal{S}_{B-I}\left(\Omega\right)$.  Each singularity
set $\Sigma\in\mathcal{S}$ satisfies the following smallness
condition:
\begin{eqnarray}
\begin{array}{ll}
\forall & U\subseteq \Omega~{\rm open}~: \\
& U\cap\Sigma~\mbox{is not residual in }U \\
\end{array}\nonumber
\end{eqnarray}
Indeed, if $\Sigma\cap U$ is residual in $U$ for some $U\subseteq
\Omega$, then $\Sigma'=U\setminus\left(U\cap \Sigma\right)$ is of
first Baire category in $\Omega$ so that $\Sigma'\in\mathcal{S}$.
Then, owing to (\ref{UnionProp}), there is some $\Sigma''\in
\mathcal{S}$ such that
\begin{eqnarray}
\Sigma\cup\Sigma'\subseteq \Sigma''.\nonumber
\end{eqnarray}
However, $U\subseteq \Sigma\cup\Sigma'$ so that $\Sigma''$ does
not have dense complement in $\Omega$, contrary to
(\ref{DenseComplement}).

In this way, we come to realize that singularity sets which
contain as particular cases the sets of first Baire category
cannot be large, even locally, in the topological sense of being
residual.  However, we should note that such sets, may still be
rather larger with respect to Lebegsgue measure.  Indeed, a closed
nowhere dense set in $\mathbb{R}^{n}$ may have arbitrary large
positive Lebesgue measure, see \cite{Oxtoby}.

\subsection{STF Algebras with Dense Singularities}

As mentioned, a key feature of the SFT algebras is the ability of
such algebras to deal with large classes of singularities of
generalized functions.  In particular, and in connection with so
called space time foam structures in general relativity, see
\cite{Mallios and Rosinger}, where the set of singular points in
the space-time manifold is dense, there is an interest in
singularities that may occur on dense sets.

The STF algebra $\mathcal{B}_{L,B-I}\left(\Omega\right)$ may, at
first, appear to provide precisely such an algebra with dense
singularity.  Indeed, the typical singularity set associated with
a generalized function in
$\mathcal{B}_{L,B-I}\left(\Omega\right)$, which is of first Baire
category, is dense in $\Omega$.  However, as shown in Section 2.4,
singularities that occur on sets of first Baire category
inevitably collapse to closed nowhere dense sets, so that
$\mathcal{A}^{\infty}_{L,nd}\left(\Omega\right)$ appears to have a
maximal position among those algebras that can handle at least all
singularities on closed nowhere dense sets.

In this section we propose an alternative
$\mathcal{B}_{L,M_{0}}\left( \Omega\right)$ to the algebra
$\mathcal{B}_{L,B-I}\left(\Omega\right)$, with the following two
features:
\begin{description}
    \item[(a)]Generalized functions in $\mathcal{B}_{L,M_{0}}\left(
    \Omega\right)$ may exhibit singularities on dense sets.
    \item[(b)] The algebra $\mathcal{B}_{L,M_{0}}\left(
    \Omega\right)$ is convenient, form the point of view of
    existence of generalized solutions of large classes of PDEs.
\end{description}
\begin{Exa}
Let $\mathcal{S}_{L,M_{0}}\left(\Omega\right)$ denote the
collection of all subsets of $\Omega$ with zero Lebesgue measure.
Recall that sets of measure $0$ may be dense in $\Omega$.  Indeed,
the set $\mathbb{Q}$ of rational numbers is dense in $\mathbb{R}$,
but being the union of countable many singleton sets, it also has
measure $0$.  Furthermore, sets of measure $0$ may not be of first
Baire category \cite{Oxtoby}.  Hence the generalized functions in
the STF algebra $\mathcal{B}_{L,M_{0}}\left(\Omega\right) =
\mathcal{C}^{\infty}\left(\Omega\right)^{\Lambda}/ \mathcal{J}_{L
,\mathcal{S}_{M_{0}}}\left(\Omega\right)$ can have singularities
on dense sets.  Furthermore, since the singularity set may not be
of first Baire category, Theorem \ref{AndFBCEq} does not apply, so
that the singularity may not collapse to a closed nowhere dense
set.

It is easy to show that the algebra $\mathcal{B}_{L,M_{0}}\left(
\Omega\right)$ admits a global version of the Cauchy-Kovalevskaia
Theorem \cite{Rosinger 14}.  Indeed, one may construct generalized
solutions of arbitrary analytic nonlinear PDEs in
$\mathcal{A}^{\infty}_{L,nd}\left(\Omega\right)$, which are
analytic everywhere except on a closed nowhere dense set.
Furthermore, the closed nowhere dense singularity set may be
chosen to have zero Lebesgue measure.  It is not hard to show that
this generalized solution belongs to $\mathcal{B}_{L,M_{0}}\left(
\Omega\right)$.
\end{Exa}

\section{Concluding Remarks}

STF algebras of generalized functions were introduced as a
convenient and natural setting in which one may deal with large
classes of singularities.  In particular, and motivated by
space-time foam structures in general relativity, algebras with
dense singularity sets were introduced, notably the algebra
$\mathcal{B}_{L,B-I}\left( \Omega\right)$.  In this paper, we
showed that $\mathcal{B}_{L,B-I}\left( \Omega\right)$ coincides
with the earlier closed nowhere dense algebra
$\mathcal{A}^{\infty}_{L,nd} \left( \Omega\right)$, which admit
only singularities that occur on closed nowhere dense sets.  From
the point of view of generalized solutions of nonlinear PDEs, this
is of great interest, since it shows that singularities in the
solutions of a PDE that occur on large sets, namely, dense sets of
first category, inevitable collapse to closed nowhere dense sets.

On the other hand, from the point of view of foam structures in
general relativity, and other highly singular phenomena, it
necessary to deal with singularities that occur on dense sets.  In
this regard, we constructed an algebra
$\mathcal{B}_{L,M_{0}}\left( \Omega\right)$ that can handle such
singularities.  Furthermore, this algebra is easily seen to be
convenient, from the point of view of solutions of large classes
of PDEs.

\end{document}